\documentclass[11pt]{amsart}

\usepackage[english]{babel}
\usepackage[utf8x]{inputenc}
\usepackage[T1]{fontenc}

\usepackage{amsmath,amsthm,amsfonts,amssymb,enumerate}
\usepackage{graphicx}
\usepackage{tikz}
\usepackage[colorlinks=true, allcolors=blue]{hyperref}
\usepackage{diagbox}

\newtheorem{theorem}{Theorem}[section]
\newtheorem{lemma}[theorem]{Lemma}
\newtheorem{corollary}[theorem]{Corollary}	
\newtheorem{proposition}[theorem]{Proposition}
\theoremstyle{definition}
\newtheorem{example}[theorem]{Example}
\newtheorem{definition}[theorem]{Definition}
\newtheorem{remark}[theorem]{Remark}

\begin{document}

\title[]{Total difference chromatic numbers of graphs}

\author{Ranjan Rohatgi}
\address{Department of Mathematics and Computer Science,
Saint Mary's College,
Notre Dame IN 46556}
\email{rrohatgi@saintmarys.edu}

\author{Yufei Zhang}
\address{Department of Mathematics and Computer Science,
Saint Mary's College,
Notre Dame IN 46556}
\email{yzhang01@saintmarys.edu}

\date{\today}

\begin{abstract}
Inspired by graceful labelings and total labelings of graphs, we introduce the idea of total difference labelings. A $k$-total labeling of a graph $G$ is an assignment of $k$ distinct labels to the edges and vertices of a graph so that adjacent vertices, incident edges, and an edge and its incident vertices receive different labels. A $k$-total difference labeling of a graph $G$ is a function $f$ from the set of edges and vertices of $G$ to the set $\{1,2,\ldots,k\}$, that is a $k$-total labeling of $G$ and for which $f(\{u,v\})=|f(u)-f(v)|$ for any two adjacent vertices $u$ and $v$ of $G$ with incident edge $\{u,v\}$. The least positive integer $k$ for which $G$ has a $k$-total difference labeling is its total difference chromatic number, $\chi_{td}(G)$. We determine the total difference chromatic number of paths, cycles, stars, wheels, gears and helms. We also provide bounds for total difference chromatic numbers of caterpillars, lobsters, and general trees.
\end{abstract}

\thanks{Y. Z. was supported by the Sister Miriam Cooney, CSC '51 Endowed Grant for Undergraduate Student Research in Mathematics}

\maketitle

\section{Introduction}\label{sec:intro}

Graph labelings have been widely studied for over half a century, as evidenced by the sheer quantity of results contained in Gallian's regularly-updated survey \cite{Gallian} of the subject. In this paper we define the \emph{total difference chromatic number} of a graph, inspired by graceful and graceful-like labelings (see, for example, \cite{Byers2} for some recent work on the subject) and total labelings. 

A graph labeling is an assignment of integers to the vertices and/or edges of a graph which follows certain conditions. For example, in a graceful labeling (introduced as a ``$\beta$-valuation'' by Rosa in \cite{Rosa} and first referred to as a ``graceful labeling'' by Golomb in \cite{Golomb}) of a graph with $m$ edges, each vertex is assigned an integer from the set $\{0,1,2,\ldots,m\}$ and each edge gets as its label the absolute value of the difference of the labels on its incident vertices. If the resulting edge labels are distinct, the graph is said to have a graceful labeling.

A (proper) \emph{$k$-labeling} of a graph is an assignment of $k$ labels to the vertices of a graph so that adjacent vertices have receive different labels. Similarly, a (proper) \emph{$k$-edge-labeling} is an assignment of $k$ labels to the edges of a graph so that incident edges receive different labels. Combining both ideas, a (proper) \emph{$k$-total-labeling} is an assignment of $k$ labels to the edges and vertices of a graph so that adjacent vertices, incident edges, and an edge and its incident vertices receive different labels. The minimum number of labels for which a $k$-labeling (respectively, $k$-edge-labeling, $k$-total-labeling) of a graph $G$ exists is called the \emph{chromatic number} (respectively, \emph{edge chromatic number}, \emph{total chromatic number}) and is denoted $\chi(G)$ (respectively, $\chi'(G)$. $\chi''(G)$). It is also common to think of labels as colors, and therefore call labelings ``colorings'' instead. As our labels are integers, we stick with the ``labeling'' terminology. Unless otherwise specified, all labelings are assumed to be proper.

Combining the mechanism of labeling the edges of a graph from graceful labelings with total labelings, we define a \emph{$k$-total difference labeling} and the \emph{total difference chromatic number} of a graph in Section~\ref{sec:prelim}, along with some preliminary observations. In Sections~\ref{sec:paths}, \ref{sec:stars}, and \ref{sec:wheels} we determine the total difference chromatic number for arbitrary paths, cycles, stars, and wheels. In the final section, we prove lower and upper bounds for the total difference chromatic numbers of caterpillars, lobsters, and general trees.

\section{Preliminaries}\label{sec:prelim}

We present a few definitions, including that of the \emph{total difference chromatic number} of a graph, and preliminary observations related to the total difference chromatic number of a graph in this section. We use $V(G)$ and $E(G)$ to denote the vertex set and edge set, respectively, of a graph $G$. For integers $a<b$ we denote by $[a,b]$ the set of integers from $a$ to $b$, inclusive. 

\begin{definition}
A \emph{labeling} of a graph $G$ is a function $c:V(G)\rightarrow \mathbb{Z}^+$ (the set of positive integers) such that $c(u)\neq c(v)$ for any two adjacent vertices $u$ and $v$. If the maximum label assigned to any vertex is $k$, we say that $c$ is a \emph{$k$-labeling} of $G$.
\end{definition}

\begin{definition}
A \emph{total labeling} of a graph $G$ is a function $f:V(G)\cup E(G)\rightarrow \mathbb{Z}^+$ such that $f(u)\neq f(v)$ for any two adjacent vertices $u$ and $v$, $f(\{u,v\})\neq f(\{v,w\})$ for any two incident edges $\{u,v\}$ and $\{v,w\}$, and $f(u)\neq f(\{u,v\})$ for any edge $\{u,v\}$ and an incident vertex $u$. If the maximum label assigned to any edge or vertex is $k$, we say that $f$ is a \emph{$k$-total labeling} of $G$.
\end{definition}

\begin{definition}\label{def:ktotdifflabel}Let $f$ be a $k$-total labeling of $G$. We say that $f$ is a \emph{$k$-total difference labeling} if $f(\{u,v\})=|f(u)-f(v)|$, whenever $\{u,v\}\in E(G)$. That is, the label assigned to an edge is the absolute difference of the labels assigned to its incident vertices.
\end{definition}

\begin{definition}\label{def:chromatictotdiff}
The \emph{total difference chromatic number} of a graph $G$, denoted $\chi_{td}(G)$, is the smallest integer $k$ for which $G$ has a $k$-total difference labeling.
\end{definition}

We first prove that the total difference chromatic number is well-defined.

\begin{proposition}\label{prop:bounds}
Given a graph $G$ with $n$ vertices, $\chi''(G)\leq \chi_{td}(G)\leq 3^{n-1}$.
\end{proposition}

\begin{proof}
The first inequality follows from the observation that a total difference labeling is precisely a total labeling with an additional condition specifying how the edges are labeled. 

To show that $\chi_{td}(G)\leq 3^{n-1}$, arbitrarily label the $n$ vertices with distinct elements from the set $\{3^0, 3^1,3^2,\ldots,3^{n-1}\}$. All edge labels will be of the form $3^i-3^j$ for distinct  integers $i,j\in[0,n-1]$ with $i>j$. Notice that $3^i-3^j=3^j(3^{i-j}-1)$, which implies that all edges labels will be distinct and different from every vertex label.
\end{proof}

\begin{example}\label{ex:bounds}
As an example to show that the total chromatic number and total difference chromatic number can be different, we consider the complete graph $K_3$. As shown in Figure~\ref{fig:k3}, $\chi''(K_3)\leq 3$.  If we attempted to totally label $K_3$ with only two colors, adjacent vertices would have the same label; hence $\chi''(K_3)=3$.

On the other hand, $\chi_{td}(K_3)=4$. Figure~\ref{fig:k3} gives 4 as an upper bound. Notice that only one of $1$ and $2$ can be used as a vertex label since all vertices are adjacent and the edge between the two vertices with these labels would also have label $1$. Therefore, 4 is also a lower bound for $\chi_{td}(K_3)$.
\end{example}

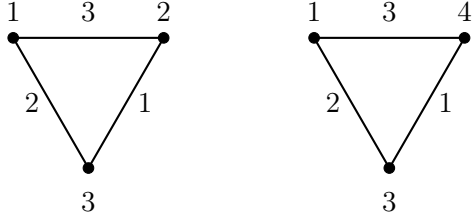
\begin{figure}
    \begin{tikzpicture}
    \draw[fill=black] (0,0) circle (2pt);
    \draw[fill=black] (2,0) circle (2pt);
    \draw[fill=black] (1,-1.732) circle (2pt);
    \draw[fill=black] (4,0) circle (2pt);
    \draw[fill=black] (6,0) circle (2pt);
    \draw[fill=black] (5,-1.732) circle (2pt);
    \draw[thick] (0,0)--(2,0)--(1,-1.732)--(0,0);
    \draw[thick] (4,0) -- (6,0)--(5,-1.732)--(4,0);
    \node at (0,0.35) {1};
    \node at (2,0.35) {2};
    \node at (1,-2.182) {3};
    \node at (1,0.35) {3};
    \node at (0.25,-0.866) {2};
    \node at (1.75,-0.866) {1};
    
    \node at (4,0.35) {1};
    \node at (6,0.35) {4};
    \node at (5,-2.182) {3};
    \node at (5,0.35) {3};
    \node at (4.25,-0.866) {2};
    \node at (5.75,-0.866) {1};
    \end{tikzpicture}
    \caption{Upper bounds for $\chi''(K_3)$ (left) and $\chi_{td}(K_3)$ (right).}
    \label{fig:k3}
\end{figure}

The following definition and propositions help us construct $k$-total difference labelings of several graphs.

\begin{definition}\label{def:doubletriple}
Let $c:V(G)\rightarrow [1,k]$ be a $k$-labeling of a graph $G$. 
\begin{enumerate}
    \item[(a)] A \emph{double} is a pair of adjacent vertices $u$ and $v$ such that $c(u)=2c(v)$. We write \emph{$(c(u),c(v))$-double} if we want to specify the labels of the vertices that form the double.
    \item[(b)] A \emph{triple} is a set of three vertices $u$,$v$,$w$ with $\{u,v\}$,$\{v,w\}\in E(G)$ such that $|c(u)-c(v)|=|c(v)-c(w)|$. We write \emph{$(c(w),c(v),c(u))$-triple} if we want to specify the labels of the vertices that form the triple.
\end{enumerate} 
See Figure~\ref{fig:doubletriple} for an example of a $(2a,a)$-double, an $(a+2b,a+b,a)$-triple and an $(a,b,a)$-triple .
\end{definition}

\begin{figure}
    \begin{tikzpicture}
    \draw[fill=black] (1,2) circle (2pt);
    \draw[fill=black] (3,2) circle (2pt);
    \draw[thick] (1,2) -- (3,2);
    \node at (1,2.35) {$a$};
    \node at (3,2.35) {$2a$};
    \node at (2,1.65) {$a$};
    
    \draw[fill=black] (0,0) circle (2pt);
    \draw[fill=black] (2,0) circle (2pt);
    \draw[fill=black] (4,0) circle (2pt);
    \draw[thick] (0,0) -- (2,0) -- (4,0);
    \node at (0,0.35) {$a$};
    \node at (2,0.35) {$a+b$};
    \node at (4,0.35) {$a+2b$};
    \node at (1,-0.35) {$|b|$};
    \node at (3,-0.35) {$|b|$};
    \draw[fill=black] (0,-2) circle (2pt);
    \draw[fill=black] (2,-2) circle (2pt);
    \draw[fill=black] (4,-2) circle (2pt);
    \draw[thick] (0,-2) -- (2,-2) -- (4,-2);
    \node at (0,-1.65) {$a$};
    \node at (2,-1.65) {$b$};
    \node at (4,-1.65) {$a$};
    \node at (1,-2.35) {$|b-a|$};
    \node at (3,-2.35) {$|b-a|$};
    \end{tikzpicture}
    \caption{In a \emph{double}, an induced edge label is equal to the label on an incident vertex. In a \emph{triple}, two incident edges receive the same label.}
    \label{fig:doubletriple}
\end{figure}
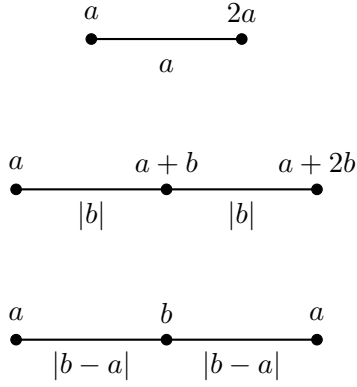

\begin{proposition}\label{prop:lowerbound}
If a graph $G$ with $n$ vertices has $\text{diam}(G)\leq 2$, then $\chi_{td}(G)\geq n$ and in any $k$-total difference labeling of $G$ all vertices must have different labels.
\end{proposition}
\begin{proof}
We prove this proposition by contradiction. Assume that a graph $G$ with $n$ vertices has $\text{diam}(G)\leq 2$ and $\chi_{td}(G)=k<n$. Then there must exist two vertices $u$ and $v$ of $G$ such that $f(u)=f(v)$, where $f$ is any $k$-total difference labeling of $G$. Since $f$ is a $k$-total difference labeling of $G$, $u$ and $v$ cannot be adjacent. Since $\text{diam}(G)\leq 2$, there must exist a vertex $w$ adjacent to both $u$ and $v$. But then $u$, $v$, and $w$ form a triple as $|f(u) - f(w)| = |f(v)-f(w)|$, contradicting that $\chi_{td}(G)<n$.
\end{proof}

\begin{lemma}
\label{lem:doublesequence}
Let $c$ be a proper $k$-labeling of a graph $G$ and let $f$ be the (not necessarily proper) total labeling of $G$ defined by $f(v)=c(v)$ for an vertex $v$ of $G$, and $f(\{u,v\})=|c(u)-c(v)|$ for an edge $\{u,v\}$ of $G$. Then $f$ is a $k$-total difference labeling of $G$ if and only if $c$ does not contain any doubles or triples.
\end{lemma}

\begin{proof}
Suppose that $f$ is a $k$-total difference labeling of $G$. By Definition~\ref{def:ktotdifflabel}, $f$ cannot assign a vertex and an incident edge the same label, nor can it assign two incident edges the same label. Hence, $c$ cannot contain a double or triple.

Now suppose that $c$ does not contain doubles or triples. We must show that $f$ does not assign the same label to adjacent vertices, a vertex and an incident edge, or two incident edges.

First, as $c$ is a proper labeling of $G$, no adjacent vertices of $G$ receive the same label. That is, $f(u)\neq f(v)$ for adjacent vertices $u$ and $v$.

Next, assume, by contradiction, that a vertex $v$ and an incident edge $\{u,v\}$ are assigned the same label by $f$. This implies that $f(v)=f(\{u,v\})=|f(u)-f(v)|$. Solving this equation for $f(u)$, we see that either $f(u)=2f(v)$ or $f(u)=0$. As $f(v)=c(v)$ for any vertex $v$ of $G$, we get that either $c(u)=2c(v)$, which contradicts the assumption that $c$ does not contain a double, or $c(u)=0$, which contradicts the assumption that $c$ is a labeling of $G$.

Finally, assume, by contradiction that for two incident edges, $\{u,v\}$ and $\{v,w\}$, we have $f(\{u,v\})=f(\{v,w\})$. An analysis similar to that in the previous case shows that $u$, $v$, and $w$ form a triple under the labeling $c$.
\end{proof}

\begin{remark}\label{rem:onlyvertexlabels}
Lemma~\ref{lem:doublesequence} provides a method of determining if a proposed $k$-total difference labeling actually is one only by looking at the vertex labels. Therefore, when constructing $k$-total difference labelings of families of graphs throughout this paper, we only provide a (proper) vertex labeling and check that it does not contain doubles or triples.
\end{remark}

We conclude this section with a result relating the total difference chromatic number of a graph to those of its subgraphs.

\begin{proposition}\label{prop:subgraph}
If $G'$ is a subgraph of $G$, then $\chi_{td}(G') \leq \chi_{td}(G)$.
\end{proposition} 
\begin{proof}
Let $f$ be a $k$-total difference labeling of $G$. As we are removing vertices or edges from $G$ to get $G'$, the restriction of $f$ to $G'$ is an $\ell$-total difference labeling of $G'$ for some $\ell\leq k$.
\end{proof}

\section{Paths and cycles}\label{sec:paths}

In this section, we determine the total difference chromatic numbers of paths and cycles. Lemma~\ref{lem:doublesequence} enables us to consider solely the labels on the vertices of these graphs.

\begin{theorem}\label{thm:path}
For any path $P_n$ with $n\geq 4,  \chi_{td}(P_n) = 4$.
\end{theorem}

\begin{proof}
We denote the vertices in $P_n$ by $v_1,v_2,\ldots,v_n$ from left to right as in Figure~\ref{fig:pathlabel}. We label $v_i$ with $1$ if $i\equiv 1\pmod{3}$, with $4$ if $i\equiv2\pmod{3}$, and with $3$ if $i\equiv0\pmod{3}$. It is straightforward to see that this labeling does not create any doubles or triples, and hence $\chi_{td}(P_n)\leq 4$.

We now show that $\chi_{td}(P_n)\geq 4$ by contradiction. Assume we can $3$-total difference label $P_n$. To avoid $(2,1)$-doubles, we must label every other vertex with $3$, in which case we form a $(3,2,3)$- or $(3,1,3)$-triple.
\end{proof}

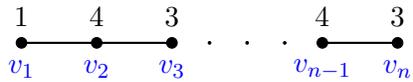
\begin{figure}
    \begin{tikzpicture}
    \draw[fill=black] (0,0) circle (2pt);
    \draw[fill=black] (1,0) circle (2pt);
    \draw[fill=black] (2,0) circle (2pt);
    \draw[fill=black] (2.5,0) circle (0.5pt);
    \draw[fill=black] (3,0) circle (0.5pt);
    \draw[fill=black] (3.5,0) circle (0.5pt);
    \draw[fill = black] (4,0) circle (2pt);
    \draw[fill=black] (5,0) circle (2pt);
    \draw[thick] (0,0)--(1,0)--(2,0);
    \draw[thick] (4,0)--(5,0);
    \node at (0,0.35) {$1$};
    \node at (1,0.35) {$4$};
    \node at (2,0.35) {$3$};
    \node at (4,0.35) {$4$};
    \node at (5,0.35) {$3$};
    \node[color=blue] at (0,-0.35) {$v_1$};
    \node[color=blue] at (1,-0.35) {$v_2$};
    \node[color=blue] at (2,-0.35) {$v_3$};
    \node[color=blue] at (4,-0.35) {$v_{n-1}$};
    \node[color=blue] at (5,-0.35) {$v_n$};
    \end{tikzpicture}
    \caption{We denote the vertices in $P_n$ by $v_1,v_2\ldots,v_n$ and provide a construction that shows $\chi_{td}(P_n)\leq 4$.}
    \label{fig:pathlabel}
\end{figure}

We now turn our attention to cycles. The total difference chromatic number of $C_n$ depends on $n$, the number of vertices in the cycle.

\begin{theorem}\label{thm:cycle}
The total difference chromatic number of a cycle is given by $\chi_{td}(C_n) = 4$ if $n\equiv 0 \pmod{3}$ and $\chi_{td}(C_n)=5$ otherwise.
\end{theorem}

\begin{proof}
Note that $C_n$ is constructed by adding an edge between the two vertices of degree 1 in $P_n$. Therefore, by Proposition~\ref{prop:subgraph}, $\chi_{td}(C_n)\geq \chi_{td}(P_n) = 4$. We denote the vertices by $v_1, v_2,\ldots,v_n$ in this cyclic order. 

If $n\equiv0\pmod3$ we can label each vertex $v_i$ in $C_n$ exactly as we labeled $v_i$ in $P_n$, and hence $\chi_{td}(C_n)=4$ in this case.

If $n\equiv1\pmod3$, we label all vertices, except $v_n$, as in the case in which $n\equiv0\pmod3$. We label $v_n$ with 5. See Figure~\ref{fig:cycle_1} for an example. To show $\chi_{td}(C_n)\geq 5$, we assume $\chi_{td}(C_n) = 4$ by contradiction. First notice that if any vertex gets label $2$, then in avoiding $(4,2)$- and $(2,1)$-doubles, we form a $(3,2,3)$-triple. Therefore, all vertices must have labels $1,3, \text{ and } 4$. Without loss of generality, suppose we label $v_1$ with 1. Then the label on $v_2$ is either $3$ or $4$ and $v_3$ gets the remaining label. To avoid doubles and triples, this sequence of labels must repeat, ending with vertex $v_{n-1}$. But then $v_n$ is forced to have a label greater than 4, completing the proof that $\chi_{td}(C_n) = 5$ when $n\equiv 1\pmod{3}$.

If $n\equiv2 \pmod3$, we label vertices $v_1,v_2,\ldots,v_{n-5}$ as in the previous two cases. We then label $v_{n-4},v_{n-3},v_{n-2},v_{n-1},v_n$ with $5,1,4,3,5$ respectively. The argument that shows $\chi_{td}(C_n)\geq 5$ for $n\equiv 2\pmod 3$ is similar to that in the $n\equiv 1\pmod 3$ case. 
 \end{proof}
 
 \begin{figure}
     \begin{tikzpicture}
     \draw[fill=black] (0,2)  circle (2pt);
     \draw[fill=black] (1,1.732) circle (2pt);
     \draw[fill=black] (1.732,1) circle (2pt);
     \draw[fill=black] (2,0) circle (2pt);
     \draw[fill=black] (0,-2)  circle (2pt);
     \draw[fill=black] (1,-1.732) circle (2pt);
     \draw[fill=black] (1.732,-1) circle (2pt);
     \draw[fill=black] (-1,-1.732) circle (2pt);
     \draw[fill=black] (-1,1.732) circle (2pt);
     \draw[fill=black] (-1.732,1) circle (2pt);
     \draw[fill=black] (-2,0) circle (2pt);
     \draw[fill=black] (-1.414,-1.414) circle (0.5pt);
     \draw[fill=black] (-1.732,-1) circle (0.5pt);
     \draw[fill=black] (-1.932,-0.5174) circle (0.5pt);
     
     \draw[thick] (-2,0) -- (-1.732,1) -- (-1,1.732) -- (0,2) -- (1,1.732) -- (1.732,1) -- (2,0) -- (1.732,-1) -- (1,-1.732) -- (0,-2)-- (-1,-1.732);
     
     \node at (0,2.3)  {1};
     \node at (1.15,1.992) {4};
     \node at (1.992,1.15) {3};
     \node at (2.3,0) {1};
     \node at (1.992,-1.15) {4};
     \node at (1.15,-1.992) {3};
     \node at (0,-2.3)  {1};
     \node at (-1.15,-1.992) {4};
     \node at (-2.3,0) {4};
     \node at (-1.992,1.15) {3};
     \node at (-1.15,1.992) {5};
     \node[color=blue] at (0,1.7) {$v_1$};
     \node[color=blue] at (0.85,1.472) {$v_2$};
     \node[color=blue] at (1.472,0.85) {$v_3$};
     \node[color=blue] at (1.7,0) {$v_4$};
     \node[color=blue] at (1.472,-0.85) {$v_5$};
     \node[color=blue] at (0.85,-1.472) {$v_6$};
     \node[color=blue] at (0,-1.7) {$v_7$};
     \node[color=blue] at (-0.85,-1.472) {$v_8$};
     \node[color=blue] at (-1.5,0) {$v_{n-2}$};
     \node[color=blue] at (-1.32,0.8) {$v_{n-1}$};
     \node[color=blue] at (-0.85,1.472) {$v_n$};
     
     \end{tikzpicture}
     \caption{A construction that shows $\chi_{td}(C_n) \leq 5$ when $n\equiv1\pmod3$.}
     \label{fig:cycle_1}
 \end{figure}
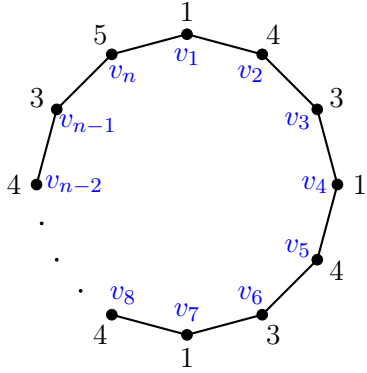

\section{Stars}\label{sec:stars}
We now consider stars, denoted $K_{1,m}$. After determining the total difference chromatic number for an arbitrary star, we explicitly determine the different labels that the maximum degree vertex of a star can receive in a $k$-total difference labeling for certain $k$. For all stars $K_{1,m}$ throughout this paper, we denote the vertex of degree $m$ by $v_0$ and the remaining vertices by $v_1,v_2,\ldots,v_m$. 
\begin{theorem}\label{thm:star}
Let $K_{1,m}$ be a star. Then
$$\chi_{td} (K_{1,m}) = \left\{\begin{array}{ll}
    m + 1, &  m \text{ is even}\\
    m +2, & m \text{ is odd} 
\end{array}\right.$$
\end{theorem}

\begin{proof}
    Consider the star $K_{1,m}$. Since $\text{diam}(K_{1,m})=2$, by Proposition~\ref{prop:lowerbound}, $\chi_{td}(K_{1,m})\geq m+1$. 
    
    We first consider the case where $m$ is even. Give $v_0$ the label $m+1$ and $v_i$ the label $i$ for $1\leq i\leq m$. Since $v_0$ has the greatest label and is odd and $v_1,\ldots,v_m$ have distinct labels, $K_{1,m}$ has no doubles or triples.
    
    If $m$ is odd, then $\chi_{td}(K_{1,m})\leq m+2$ if we label $v_0$ with $m+2$ and the rest as in the even case. Note that there are no doubles or triples using the same argument as before. (See Figure~\ref{fig:K_1,m_even} for the construction.) 
    
    To prove $\chi_{td}(K_{1,m}) \geq m+2$, we assume $\chi_{td}(K_{1,m}) = m+1$ by contradiction. In this case, each label in $[1,m+1]$ appears on exactly one $v_i$. Note that we cannot use any integer in $[2,m]$ to label $v_0$: in this case there must be two leaves whose labels, along with the label for $v_0$, will form a $(\ell+1,\ell,\ell-1)$-triple for some $\ell\in [2,m]$. We also cannot use $1$ or $m+1$ to label $v_0$ because we get a $(2,1)$- or $(m+1,\frac{m+1}{2})$-double. Thus, when $m$ is odd, $\chi_{td}(K_{1,m}) = m+2$.

\end{proof}

\begin{figure}
    \begin{tikzpicture}
    \draw[fill=black] (0,0) circle (2pt);
    \draw[fill=black] (3,2) circle (2pt);
    \draw[fill=black] (3,1) circle (2pt);
    \draw[fill=black] (3,0) circle (2pt);
    \draw[fill=black] (3,-0.5) circle (0.5pt);
    \draw[fill=black] (3,-1) circle (0.5pt);
    \draw[fill=black] (3,-1.5) circle (0.5pt);
    \draw[fill=black] (3,-2) circle (2pt);
    
    \draw[thick] (3,2) -- (0,0) --(3,1) -- (0,0) -- (3,0) -- (0,0) -- (3,-2);
    
    \node at (-0.7,0.35) {$m+1$};
    \node at (-0.7,0) {or};
    \node at (-0.7,-0.35) {$m+2$};
    \node at (3.3,2) {1};
    \node at (3.3,1) {2};
    \node at (3.3,0) {3};
    \node at (3.3,-2) {$m$};
    \end{tikzpicture}
    \caption{If $m$ is even then $\chi_{td}(K_{1,m})\leq m+1$, and if $m$ is odd then $\chi_{td}(K_{1,m})\leq m+2$.}
    \label{fig:K_1,m_even}
\end{figure}
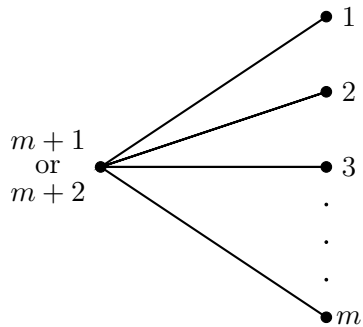

\begin{remark}
One can verify that if $m$ is even, $v_0$ must receive the label $m+1$, and that if $m$ is odd, $v_0$ must receive the label $1$ or $m+2$.
\end{remark}

We now determine the possible labels for $v_0$ in an $(m+r)$-total difference labeling of $K_{1,m}$ for $1\leq r\leq m$. We will use this result in Section~\ref{sec:trees} to provide an upper bound for the total difference chromatic number of general trees.

\begin{lemma}\label{lem:star-realize}
Consider an $(m+r)$-total difference labeling of the star $K_{1,m}$ with $1\leq r\leq m$. Then $v_0$ can have any label in the set $[1,r-1]\cup[m+2,m+r]$. If $m$ is even or if $m$ is odd and $r=\frac{m+3}{2}$, $v_0$ can additionally receive the label $m+1$.

\end{lemma}
\begin{proof}
Throughout this proof, we let $\ell$ be the label given to $v_0$. The greatest label on a vertex in $K_{1,m}$ must be $m+r$ since we are considering $(m+r)$-total difference labelings of $K_{1,m}$. Since $K_{1,m}$ has $m+1$ vertices, which require $m+1$ distinct labels from $[1,m+r]$, we must use all but $r-1$ integers in $[1,m+r-1]$. We look at four cases, based on the possible values of $\ell$, namely: $\ell  \in [1,r-1]$, $\ell  = r$, $\ell \in [r+1,m]$, and $\ell \in [m+1, m+r]$. Note that if $r=1$, we have only three distinct cases.

If $1\leq\ell\leq r-1$, then $\ell$ could be the middle value in the $\ell-1$ triples of the form $(\ell+i,\ell,\ell-i)$ where $1\leq i\leq \ell-1$. Additionally, $\ell$ could form a double with $2\ell$ or with $\frac{\ell}{2}$ if $\ell$ is even.  Therefore there are $\ell$ doubles and triples which include $\ell$ but are disjoint otherwise. (Notice that the $(\ell,\frac{\ell}{2})$-double is not disjoint with the $(\frac{3\ell}{2},\ell,\frac{\ell}{2})$-triple). Since $\ell\leq r-1$, we must avoid using at most $r-1$ integers, as desired. Therefore, $v_0$ can be labeled with any integer in $[1,r-1]$.

Suppose $\ell=r$, then again $\ell$ could be the middle value in the same $\ell-1$ triples (of the form $(\ell+i,\ell, \ell-i)$ as above) with the integers in $[1,m+r]$, 
and a double with $ 2\ell = 2r \leq m+r$. Therefore, we have $\ell=r$ labels that cannot be used, giving us only $m$ possible labels for $m+1$ vertices. So, $r$ cannot be used to label $v_0$. 

 Assume $r+1\leq \ell \leq m$. Then $\ell$ could be the middle value in $r$ triples (and perhaps more) with the integers in $[1,m+r]$, namely, $(\ell+i,\ell,\ell-i)$ where $1\leq i\leq r$. As in the previous case, $v_0$ cannot be labeled using any element of $[r+1,m]$.

If $m+1 \leq \ell \leq m+r$ there are $m+r-\ell$ triples possible: $(\ell+i,\ell,\ell-i)$ for $i\in[1,m+r-\ell]$. Therefore, if $\ell = m+ j$ for $j\in[1,r]$, there are $r-j$ possible triples. In addition to these triples, $\ell$ may also be part of a double. Therefore we must avoid $r-j+1$ labels for $j\in[1,r]$. If $j>1$ then we are able to label $v_0$ with $m+j$ as there are at most $r-1$ forbidden labels. If $j=1$ and $m$ is even, $\ell$ could not be part of a double as it is odd, and hence $\ell$ could be $m+1$. If $m$ is odd we could have up to $r$ forbidden labels ($r-1$ triples and a double) and so $\ell$ could not be $m+1$. In all of these cases except when $r=\frac{m+3}{2}$, the triples and double are disjoint and hence $\ell$ cannot be $m+1$. The exception is described in the next paragraph.

Consider the case where $m$ is odd and $r=\frac{m+3}{2}$. Then for $\ell = m+1$, we will have $r-1$ triples of the form ($\ell+i,\ell,\ell-i$) with $i$ in $[1,r-1]$, or equivalently in $[1,\frac{m+1}{2}]$. We now consider doubles. Notice that the $(\ell,\frac{\ell}{2})$-double (equivalently, $(m+1,\frac{m+1}{2})$-double) and the triple $(\ell+i,\ell,\ell-i)$ with $i=\frac{m+1}{2}$ both contain the label $\frac{m+1}{2}$. We can avoid both the double and triple simply by not using the label $\frac{m+1}{2}$, hence the possible double does not give us an additional label that must be avoided. It is also easy to check that the $(2\ell,\ell)$-double is not possible since $2\ell=2m+2>\frac{3m+3}{2}=m+\frac{m+3}{2}=m+r$ exceeds our bound. Hence, $v_0$ can have label $m+1$ as there are only $r-1$ possible violations of Lemma~\ref{lem:doublesequence}.
\end{proof}

\section{Wheels, gears, and helms}\label{sec:wheels}

We create the wheel, $W_n$, from the cycle $C_{n-1}$ by adding a vertex adjacent to all other vertices in the cycle. We use Theorem~\ref{thm:star} to determine the total difference chromatic number of wheels. 

\begin{theorem}\label{thm:wheels}
For $n\geq 4$, $\chi_{td}(W_n)=\chi_{td}(K_{1,n})$ except when $n$ is $4$ or $5$. Explicitly,
\[ \chi_{td}(W_n)=\begin{cases}
8 & n=4 \\
    7 & n=5 \\
    n+1 & n \text{ is even and } n\geq 6\\
    n &  n \text{ is odd and } n\geq 7.
\end{cases}\]
\end{theorem}

\begin{proof}
We denote by $v_0$ the vertex with degree $n-1$ and the remainder of the vertices by $v_1,\ldots,v_{n-1}$ in this cyclic order. Notice that $K_{1,n-1}$ is a subgraph of $W_n$ so Proposition~\ref{prop:subgraph} implies $\chi_{td}(W_n) \geq \chi_{td}(K_{1,n-1})$.

Suppose $n=4$. Note that in $W_4$, all vertices are adjacent to each other, so $W_4 = K_4$. We obtain a total difference labeling of $W_4$ by labeling the vertices with $1,5,7,8$ (see the graph in the top left in Figure~\ref{fig:w4w5w6w7}). It is straightforward to show that $\chi_{td}(W_4)\geq 8$ by case analysis. 

For $W_5$, we label $v_0$ with $7$ and $v_1,v_2,v_3,v_4$ with $1,3,2,5$, respectively. The graph in the top right in Figure~\ref{fig:w4w5w6w7} shows this construction. 

Since $W_5$ has diameter 2, we use Proposition~\ref{prop:lowerbound} to see that $5 \leq \chi_{td}(W_5)\leq 7$. We will now show that it cannot be 5 or 6.
Suppose that $\chi_{td}(W_5) = 5$. If $v_0$ is labeled with 2, 3, or 4, then a triple must be formed. If $v_0$ gets the label 1, it must be adjacent to the vertex that gets labeled with 2. Finally, if $v_0$ is labeled with 5, one of the vertices with label 1 or 4 must be adjacent to the vertex with label 2. We can show similarly that $\chi_{td}(W_5)\neq 6$ by ruling out possible labels for $v_0$. 

\begin{figure}
    \begin{tikzpicture}
    \draw[fill=black] (0,0) circle (2pt);
    \draw[fill=black] (0,1.6) circle (2pt);
    \draw[fill=black] (-1.5,-1)  circle (2pt);
    \draw[fill=black] (1.5,-1) circle (2pt);

    \draw[thick] (0,1.6) -- (0,0) -- (-1.5,-1) -- (0,1.6) -- (1.5,-1) -- (0,0) -- (1.5,-1) -- (-1.5,-1);
    
    \node at (0,-0.25) {8};
    \node at (0,1.9) {1};
    \node at (-1.8,-1) {7};
    \node at (1.8,-1) {5};

    \draw[fill=black] (6,0.3) circle (2pt);
    \draw[fill=black] (4.7,1.6) circle (2pt);
    \draw[fill=black] (7.3,1.6) circle (2pt);
    \draw[fill=black] (4.7,-1) circle (2pt);
    \draw[fill=black] (7.3,-1) circle (2pt);

    \draw[thick] (4.7,1.6)--(7.3,1.6) -- (7.3,-1) -- (4.7,-1) -- (4.7,1.6) -- (6,0.3) -- (4.7,-1) -- (7.3,1.6) -- (6,0.3) -- (7.3,-1);
    \node at (6,0.6)  {7};
    \node at (4.7,1.9) {1};
    \node at (7.3,1.9) {3};
    \node at (7.3,-1.3) {2};
    \node at (4.7,-1.3) {5};
    
    \draw[fill=black] (0,-4) circle (2pt);
    \draw[fill=black] (0,-2.5) circle (2pt);
    \draw[fill=black] (1.5,-3.6) circle (2pt);
    \draw[fill=black] (-1.5,-3.6) circle (2pt);
    \draw[fill=black] (-1,-5.3) circle (2pt);
    \draw[fill=black] (1,-5.3) circle (2pt);

    \draw[thick] (0,-4)--(0,-2.5)--(1.5,-3.6)--(1,-5.3)--(-1,-5.3)--(-1.5,-3.6)--(0,-2.5);
    \draw[thick] (1.5,-3.6)--(0,-4)--(1,-5.3);
    \draw[thick] (-1,-5.3)--(0,-4)--(-1.5,-3.6);
    \node at (0.2,-3.7) {1};
    \node at (0,-2.25) {3};
    \node at (1.75,-3.6) {4};
    \node at (-1.75,-3.6) {7};
    \node at (-1,-5.6) {5};
    \node at (1,-5.6) {6};

    \draw[fill=black] (6,-4) circle (2pt);
    \draw[fill=black] (7.5,-4) circle (2pt);
    \draw[fill=black] (4.5,-4) circle (2pt);
    \draw[fill=black] (6.75,-2.7) circle (2pt);
    \draw[fill=black] (5.25,-2.7) circle (2pt);
    \draw[fill=black] (6.75,-5.3) circle (2pt);
    \draw[fill=black] (5.25,-5.3) circle (2pt);

    \draw[thick] (6,-4)--(7.5,-4)--(6.75,-5.3)--(5.25,-5.3)--(4.5,-4)--(5.25,-2.7)--(6.75,-2.7)--(7.5,-4);
    \draw[thick](6.75,-5.3)--(6,-4)--(5.25,-5.3);
    \draw[thick] (4.5,-4)--(6,-4) -- (5.25,-2.7);
    \draw[thick] (6,-4) -- (6.75,-2.7);
    \node at (6.35,-3.8) {7};
    \node at (7.75,-4) {3};
    \node at (4.25,-4) {4};
    \node at (6.75,-2.4) {2};
    \node at (5.25,-2.4) {6};
    \node at (6.75,-5.6) {1};
    \node at (5.25,-5.6) {5};

    \end{tikzpicture}
    \caption{A construction that shows $\chi_{td}(W_4)\leq 8$ and $\chi_{td}(W_5), \chi_{td}(W_6), \chi_{td}(W_7) \leq 7$.}
    \label{fig:w4w5w6w7}
\end{figure}
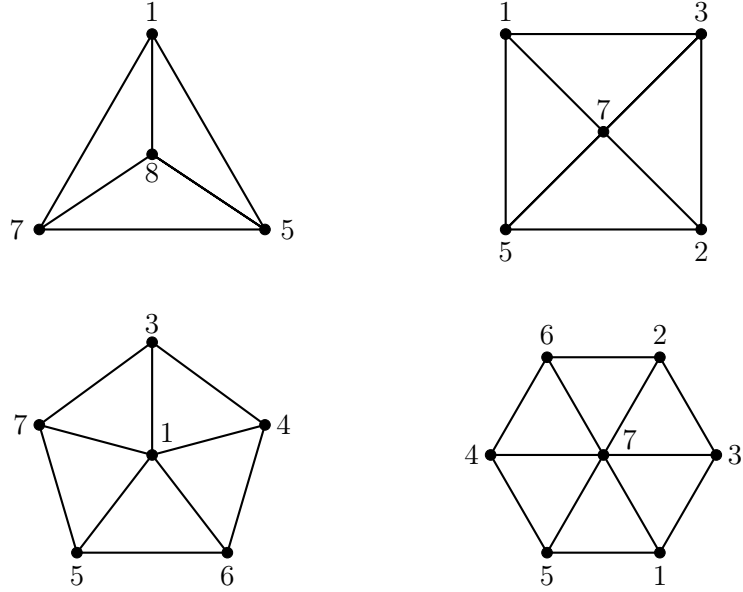
For $\chi_{td}(W_6)$ and $\chi_{td}(W_7)$, constructions are provided in the bottom two graphs in Figure~\ref{fig:w4w5w6w7}. Notice that we do indeed get that $\chi_{td}(W_6)=\chi_{td}(W_7)=7$ due to Theorem~\ref{thm:star} about the total difference chromatic number of stars and Proposition~\ref{prop:subgraph}.

We construct total difference labelings of $W_n$ for the general case, in which $n\geq 8$. If $n$ is odd, we label $v_0$ with $n$ and the vertices $v_1, v_3, v_5,\ldots,v_{n-2}$ with the consecutive odd integers $1, 3, 5, \ldots, n-2$, respectively. We label $v_2, v_4, v_6, \ldots, v_{n-1}$ with $n-1, 2, 4, 6, \ldots, n-3$, respectively. See Figure~\ref{fig:w_odd} for the construction.

Notice that, except for vertices $v_2$ and $v_{n-1}$, vertices with even labels are adjacent to vertices odd labels greater than their own. This immediately shows that these vertices cannot be in a double, and can be used to show, along with Proposition~\ref{prop:lowerbound}, that none of these vertices can be in a triple. It is then straightforward to check the remaining cases: that no triple can involve $v_2$, $v_{n-1}$, or only vertices with odd labels.

 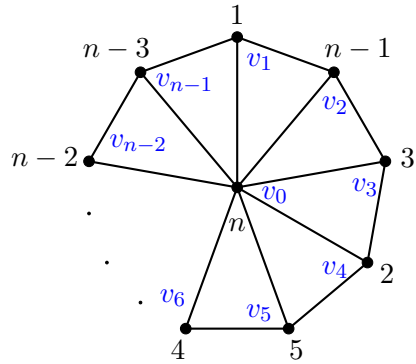
\begin{figure}
     \begin{tikzpicture}
     \draw[fill=black] (0,0)  circle (2pt);
     \draw[fill=black] (0,2)  circle (2pt);
     \draw[fill=black] (1.2856,1.532)  circle (2pt);
     \draw[fill=black] (1.97,0.347)  circle (2pt);
     \draw[fill=black] (1.732,-1)  circle (2pt);
     \draw[fill=black] (0.684,-1.879)  circle (2pt);
     \draw[fill=black] (-0.684,-1.879)  circle (2pt);
     \draw[fill=black] (-1.732,-1)  circle (0.5pt);
     \draw[fill=black] (-1.2856,-1.532)  circle (0.5pt);
     \draw[fill=black] (-1.97,-0.347)  circle (0.5pt);
     \draw[fill=black] (-1.97,0.347)  circle (2pt);
     \draw[fill=black] (-1.2856,1.532)  circle (2pt);
     
     \draw[thick] (0,2)--(1.2856,1.532)--(1.2856,1.532)--(1.97,0.347)--(1.732,-1)--(0.684,-1.879)--(-0.684,-1.879);
     \draw[thick] (0,2)--(-1.2856,1.532)--(-1.97,0.347);
     \draw[thick] (0,0)--(0,2);
     \draw[thick] (0,0)--(1.2856,1.532);
     \draw[thick] (0,0)--(1.97,0.347);
     \draw[thick] (0,0)--(1.732,-1);
     \draw[thick] (0,0)--(0.684,-1.879);
     \draw[thick] (0,0)--(-0.685,-1.879);
     \draw[thick] (0,0)--(-1.2856,1.532);
     \draw[thick] (0,0)--(-1.97,0.347);

     \node[blue] at (0.5,-0.1) {$v_0$};
     \node[blue] at (0.3,1.674) {$v_1$};
     \node[blue] at (1.3,1.093) {$v_2$};
     \node[blue] at (1.7,0) {$v_3$};
     \node[blue] at (1.3,-1.093) {$v_4$};
     \node[blue] at (0.295,-1.67) {$v_5$};
     \node[blue] at (-0.85,-1.47) {$v_6$};
     \node[blue] at (-1.3155,0.581) {$v_{n-2}$};
     \node[blue] at (-0.7,1.386) {$v_{n-1}$};
     
     \node at (0,-0.5) {$n$};
     \node at (0,2.3) {$1$};
     \node at (2.265,0.399) {$3$};
     \node at (0.7866,-2.1613) {$5$};
     \node at (-2.56,0.3994) {$n-2$};
     \node at (1.607,1.839) {$n-1$};
     \node at (1.992,-1.15) {$2$};
     \node at (-0.7866,-2.1613) {$4$};
     \node at (-1.607,1.838) {$n-3$};
     
     \end{tikzpicture}
     \caption{A construction that shows $\chi_{td}(W_n) \leq n$ when $n$ is odd and $n\geq7$.}
     \label{fig:w_odd}
 \end{figure}

If $n$ is even, our construction is nearly identical: vertices $v_1,v_3,v_5,\ldots,v_{n-1}$ are labeled with $1,3,5,\ldots,n-1$, respectively and $v_2,v_4,v_6,\ldots,v_{n-2}$ with $n-2, 2, 4, 6,\ldots,n-4$. The only difference is that $v_0$ gets the label $n+1$. Appealing again to Theorem~\ref{thm:star} and Proposition~\ref{prop:subgraph} completes the proof. 
\end{proof}

We turn our attention to two families of graphs related to wheels: gears and helms. For $n\geq 4$, the gear $G_n$ is formed by subdividing each edge on the outer circuit of $W_n$ into two edges. The helm $H_n$ is created from $W_n$ by appending a leaf to each vertex on the outer circuit of $W_n$. See Figure~\ref{fig:g4g5g6g7} for several examples of gears and Figure~\ref{fig:h7} for an example of a helm.

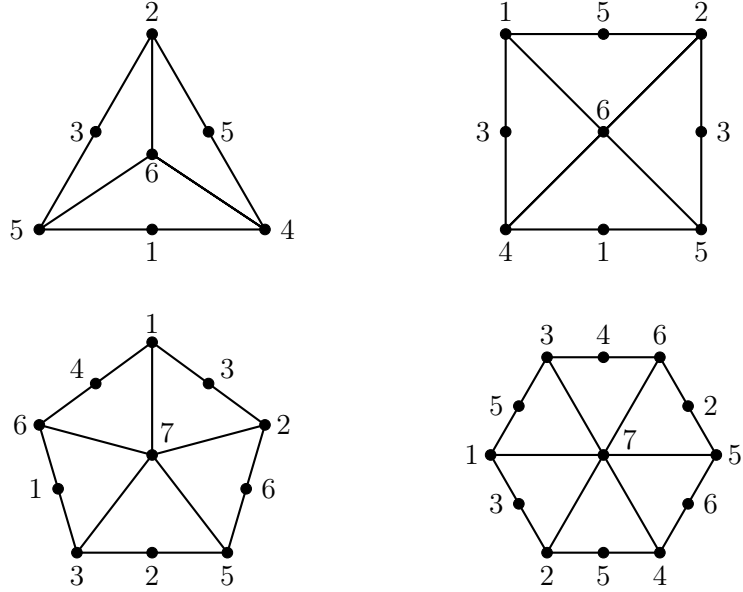
\begin{figure}
    \begin{tikzpicture}
    \draw[fill=black] (0,0) circle (2pt);
    \draw[fill=black] (0,1.6) circle (2pt);
    \draw[fill=black] (-1.5,-1)  circle (2pt);
    \draw[fill=black] (1.5,-1) circle (2pt);
    \draw[fill=black] (-0.75,0.3) circle (2pt);
    \draw[fill=black] (0.75,0.3) circle (2pt);
    \draw[fill=black] (0,-1) circle (2pt);
    
    \draw[thick] (0,1.6) -- (0,0) -- (-1.5,-1) -- (0,1.6) -- (1.5,-1) -- (0,0) -- (1.5,-1) -- (-1.5,-1);
    
    \node at (0,-0.25) {6};
    \node at (0,1.9) {2};
    \node at (-1.8,-1) {5};
    \node at (1.8,-1) {4};
    \node at (-1,0.3) {3};
    \node at (1,0.3) {5};
    \node at (0,-1.3) {1};

    \draw[fill=black] (6,0.3) circle (2pt);
    \draw[fill=black] (4.7,1.6) circle (2pt);
    \draw[fill=black] (7.3,1.6) circle (2pt);
    \draw[fill=black] (4.7,-1) circle (2pt);
    \draw[fill=black] (7.3,-1) circle (2pt);
    \draw[fill=black] (6,1.6) circle (2pt);
    \draw[fill=black] (6,-1) circle (2pt);
    \draw[fill=black] (7.3,0.3) circle (2pt);
    \draw[fill=black] (4.7,0.3) circle (2pt);
    
    \draw[thick] (4.7,1.6)--(7.3,1.6) -- (7.3,-1) -- (4.7,-1) -- (4.7,1.6) -- (6,0.3) -- (4.7,-1) -- (7.3,1.6) -- (6,0.3) -- (7.3,-1);
    \node at (6,0.6)  {6};
    \node at (4.7,1.9) {1};
    \node at (7.3,1.9) {2};
    \node at (7.3,-1.3) {5};
    \node at (4.7,-1.3) {4};
    \node at (6,1.9) {5};
    \node at (6,-1.3) {1};
    \node at (4.4,0.3) {3};
    \node at (7.6,0.3) {3};
    
    \draw[fill=black] (0,-4) circle (2pt);
    \draw[fill=black] (0,-2.5) circle (2pt);
    \draw[fill=black] (1.5,-3.6) circle (2pt);
    \draw[fill=black] (-1.5,-3.6) circle (2pt);
    \draw[fill=black] (-1,-5.3) circle (2pt);
    \draw[fill=black] (1,-5.3) circle (2pt);
    \draw[fill=black] (0,-5.3) circle (2pt);
    \draw[fill=black] (-1.25,-4.45) circle (2pt);
    \draw[fill=black] (1.25,-4.45) circle (2pt);
    \draw[fill=black] (0.75,-3.05) circle (2pt);
    \draw[fill=black] (-0.75,-3.05) circle (2pt);
    \draw[thick] (0,-4)--(0,-2.5)--(1.5,-3.6)--(1,-5.3)--(-1,-5.3)--(-1.5,-3.6)--(0,-2.5);
    \draw[thick] (1.5,-3.6)--(0,-4)--(1,-5.3);
    \draw[thick] (-1,-5.3)--(0,-4)--(-1.5,-3.6);
    \node at (0.2,-3.7) {7};
    \node at (0,-2.25) {1};
    \node at (1.75,-3.6) {2};
    \node at (-1.75,-3.6) {6};
    \node at (-1,-5.6) {3};
    \node at (1,-5.6) {5};
    \node at (0,-5.6) {2};
    \node at (1.55,-4.45) {6};
    \node at (-1.55,-4.45) {1};
    \node at (1,-2.85) {3};
    \node at (-1,-2.85) {4};
    
    \draw[fill=black] (6,-4) circle (2pt);
    \draw[fill=black] (7.5,-4) circle (2pt);
    \draw[fill=black] (4.5,-4) circle (2pt);
    \draw[fill=black] (6.75,-2.7) circle (2pt);
    \draw[fill=black] (5.25,-2.7) circle (2pt);
    \draw[fill=black] (6.75,-5.3) circle (2pt);
    \draw[fill=black] (5.25,-5.3) circle (2pt);
    \draw[fill=black] (6,-2.7) circle (2pt);
    \draw[fill=black] (6,-5.3) circle (2pt);
    \draw[fill=black] (7.125,-4.65) circle (2pt);
    \draw[fill=black] (4.875,-4.65) circle (2pt);
    \draw[fill=black] (7.125,-3.35) circle (2pt);
    \draw[fill=black] (4.875,-3.35) circle (2pt);
    
    \draw[thick] (6,-4)--(7.5,-4)--(6.75,-5.3)--(5.25,-5.3)--(4.5,-4)--(5.25,-2.7)--(6.75,-2.7)--(7.5,-4);
    \draw[thick](6.75,-5.3)--(6,-4)--(5.25,-5.3);
    \draw[thick] (4.5,-4)--(6,-4) -- (5.25,-2.7);
    \draw[thick] (6,-4) -- (6.75,-2.7);
    \node at (6.35,-3.8) {7};
    \node at (7.75,-4) {5};
    \node at (4.25,-4) {1};
    \node at (6.75,-2.4) {6};
    \node at (5.25,-2.4) {3};
    \node at (6.75,-5.6) {4};
    \node at (5.25,-5.6) {2};
    \node at (6,-2.4) {4};
    \node at (6,-5.6) {5};
    \node at (7.425,-3.35) {2};
    \node at (4.575,-3.35) {5};
    \node at (7.425,-4.65) {6};
    \node at (4.575,-4.65) {3};
    
    \end{tikzpicture}
    \caption{Constructions that show $\chi_{td}(G_4),  \chi_{td}(G_5)\leq6$ and $\chi_{td}(G_6), \chi_{td}(G_7)\leq8$.}
    \label{fig:g4g5g6g7}
\end{figure}

\begin{theorem}
For gears $G_n$ with $n\geq 4$, \[\chi_{td}(G_n) = \left\{\begin{array}{ll}
    6 & n=4,5\\
    n+1 & n \text{ is even and } n\geq 6\\
    n &  n \text{ is odd and } n\geq 7. 
\end{array}\right.\]
\end{theorem}

\begin{proof}
The constructions in Figure~\ref{fig:g4g5g6g7} show that the theorem holds when $4\leq n\leq 7$. Observe that $K_{1,n-1}$ is a subgraph of $G_n$, so that $\chi_{td}(G_n)\geq\chi_{td}(K_{1,n-1})$ by Proposition~\ref{prop:subgraph}. Theorem~\ref{thm:star} implies that $\chi_{td}(K_{1,n-1})=n$ if $n$ is odd and $\chi_{td}(K_{1,n-1})=n+1$ if $n$ is even so it remains to show by construction that $\chi_{td}(G_n)\leq\chi_{td}(K_{1,n-1})$ for $n\geq 8$. We break up the construction into two cases based upon the parity of $n$.   

First suppose that $n$ is even and $n\geq 8$. We let $v_0$ be the vertex of degree $n-1$ in $G_n$. We denote the remaining vertices by $v_1,v_2,\ldots,v_{2n-2}$ in this cyclic order, choosing $v_1$ to be any vertex adjacent to $v_0$. Notice that $v_{2i}$ has degree 2 while $v_{2i-1}$ has degree 3 for $1\leq i\leq n-1$. We label $v_0$ with $n+1$ and $v_1,v_3,\ldots,v_{2n-3}$ with $1,2,\ldots, n-1$, respectively. We then label $v_2,v_4,\ldots, v_{2n-2}$ with $n-2,n-1,5,6,\ldots,n-1,2,3$, respectively. It is straightforward to check that this does indeed give an $(n+1)$-total difference labeling of $G_n$ when $n\geq 8$ is even. This construction is shown in Figure~\ref{fig:g_even}.

If $n$ is odd and $n\geq 9$, we use a similar construction to show that $\chi_{td}(G_n)=n$: the only difference is that $v_0$ gets label $n$ rather than $n+1$. 

 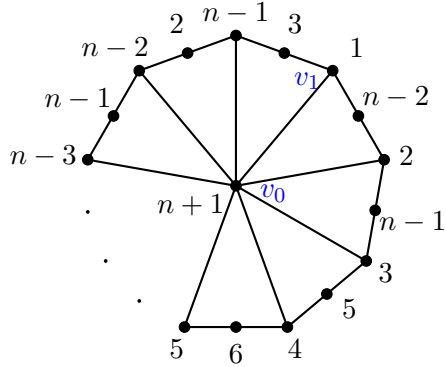
\begin{figure}
     \begin{tikzpicture}
     \draw[fill=black] (0,0)  circle (2pt);
     \draw[fill=black] (0,2)  circle (2pt);
     \draw[fill=black] (1.2856,1.532)  circle (2pt);
     \draw[fill=black] (1.97,0.347)  circle (2pt);
     \draw[fill=black] (1.732,-1)  circle (2pt);
     \draw[fill=black] (0.684,-1.879)  circle (2pt);
     \draw[fill=black] (-0.684,-1.879)  circle (2pt);
     \draw[fill=black] (-1.732,-1)  circle (0.5pt);
     \draw[fill=black] (-1.2856,-1.532)  circle (0.5pt);
     \draw[fill=black] (-1.97,-0.347)  circle (0.5pt);
     \draw[fill=black] (-1.97,0.347)  circle (2pt);
     \draw[fill=black] (-1.2856,1.532)  circle (2pt);
     \draw[fill=black] (0.6428,1.766)  circle (2pt);
     \draw[fill=black] (1.6278,0.9295)  circle (2pt);
     \draw[fill=black] (1.851,-0.3265)  circle (2pt);
     \draw[fill=black] (1.208,-1.4395)  circle (2pt);
     \draw[fill=black] (0,-1.879)  circle (2pt);
     \draw[fill=black] (-0.6428,1.766)  circle (2pt);
     \draw[fill=black] (-1.6278,0.9295)  circle (2pt);
     
     \draw[thick] (0,2)--(1.2856,1.532)--(1.2856,1.532)--(1.97,0.347)--(1.732,-1)--(0.684,-1.879)--(-0.684,-1.879);
     \draw[thick] (0,2)--(-1.2856,1.532)--(-1.97,0.347);
     \draw[thick] (0,0)--(0,2);
     \draw[thick] (0,0)--(1.2856,1.532);
     \draw[thick] (0,0)--(1.97,0.347);
     \draw[thick] (0,0)--(1.732,-1);
     \draw[thick] (0,0)--(0.684,-1.879);
     \draw[thick] (0,0)--(-0.685,-1.879);
     \draw[thick] (0,0)--(-1.2856,1.532);
     \draw[thick] (0,0)--(-1.97,0.347);

     \node[blue] at (0.5,-0.1) {$v_0$};
     \node[blue] at (0.95,1.4) {$v_1$};
     
     \node at (-0.6,-0.25) {$n+1$};
     \node at (1.607,1.839) {$1$};
     \node at (2.265,0.399) {$2$};
     \node at (1.992,-1.15) {$3$};
     \node at (0.7866,-2.1613) {$4$};
      \node at (-0.7866,-2.1613) {$5$};
     \node at (-2.56,0.3994) {$n-3$};
     \node at (-1.607,1.838) {$n-2$};
     \node at (0,2.3) {$n-1$};
     
     \node at (2.1278,1.1795) {$n-2$};
     \node at (2.351,-0.4865) {$n-1$};
     \node at (1.508,-1.6395) {$5$};
     \node at (0,-2.229) {$6$};
     \node at (-2.1,1.15) {$n-1$};
     \node at (-0.7866,2.1613) {$2$};
     \node at (0.7866,2.1613) {$3$};
     \end{tikzpicture}
     \caption{Vertices $v_0$ and $v_1$ (with labels $n$ and $1$, respectively) are named in the figure. Vertex $v_2$ has label $n-2$ and is adjacent to $v_1$, and $v_3,v_4,\ldots,v_{2n-2}$ continue clockwise. This construction shows $\chi_{td}(G_n) \leq n+1$ when $n$ is even and $n\geq8$.}
     \label{fig:g_even}
 \end{figure}
\end{proof}

\begin{theorem}\label{thm:helms}
For $n\geq 4$, $\chi_{td}(H_n)=\chi_{td}(W_n)$ except when $n$ is $6$ or $7$. In these cases, $\chi_{td}(H_6)=\chi_{td}(H_7)=8$.
\end{theorem}

We do not prove Theorem \ref{thm:helms} in detail. By Proposition~\ref{prop:subgraph} it is clear that $\chi_{td}(H_n)\geq\chi_{td}(W_n)$. For $n\geq 8$, a construction that shows $\chi_{td}(H_n)\leq\chi_{td}(W_n)$ can be obtained from the total difference labelings for $W_n$ described in Theorem~\ref{thm:wheels} by determining labels for the leaves of $H_n$ that avoid doubles and triples.

\begin{figure}
    \centering
    \begin{tikzpicture}[scale=1]
    \draw[fill=black] (6,0) circle (2pt);
    \draw[fill=black] (5.25,1.3) circle (2pt);
    \draw[fill=black] (4.8,2.1) circle (2pt);
    \draw[fill=black] (4.8,-2.1) circle (2pt);
    \draw[fill=black] (7.2,2.1) circle (2pt);
    \draw[fill=black] (7.2,-2.1) circle (2pt);
    \draw[fill=black] (3.5,0) circle (2pt);
    \draw[fill=black] (7.5,0) circle (2pt);
    \draw[fill=black] (8.5,0) circle (2pt);
    \draw[fill=black] (4.5,0) circle (2pt);
    \draw[fill=black] (6.75,1.3) circle (2pt);
    \draw[fill=black] (6.75,-1.3) circle (2pt);
    \draw[fill=black] (5.25,-1.3) circle (2pt);
    \draw[thick] (6,0)--(7.5,0)--(6.75,-1.3)--(5.25,-1.3)--(4.5,0)--(5.25,1.3)--(6.75,1.3)--(7.5,0)--(8.5,0);
    \draw[thick](7.2,-2.1)--(6.75,-1.3)--(6,0)--(5.25,-1.3)--(4.8,-2.1);
    \draw[thick] (3.5,0)--(4.5,0)--(6,0) -- (5.25,1.3)--(4.8,2.1);
    \draw[thick] (6,0) -- (6.75,1.3)--(7.2,2.1);
    \node at (6.35,0.2) {8};
    \node at (7.75,0.2) {1};
    \node at (4.25,0.2) {2};
    \node at (6.75,1.6) {6};
    \node at (5.25,1.6) {5};
    \node at (6.75,-1.6) {7};
    \node at (5.25,-1.6) {3};
    \node at (3.3,0) {7};
    \node at (8.8,0) {3};
    \node at (4.65,2.3) {7};
    \node at (4.65,-2.3) {5};
    \node at (7.4,2.3) {2};
    \node at (7.4,-2.3) {2};
    
    \end{tikzpicture}
    \caption{A construction that shows $\chi_{td}(H_7) \leq 7$.}
    \label{fig:h7}
\end{figure}
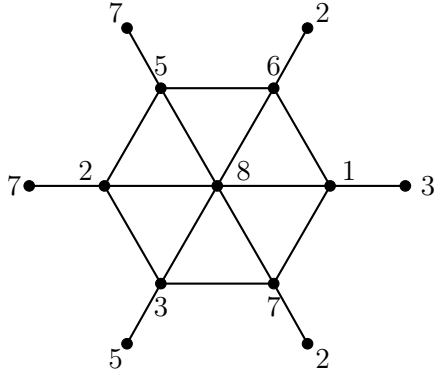

\section{Trees}\label{sec:trees}
In this section, we first determine the total difference chromatic numbers of caterpillars.  We then provide lower and upper bounds for the total difference chromatic numbers of lobsters, and end with an upper bound for the total difference chromatic numbers of general trees.

\begin{definition}
A \emph{caterpillar} is a tree with the property that removal of its vertices of degree one and their incident edges leaves a path.
\end{definition}

Consider the path formed by the removal of the degree one vertices in a given caterpillar. The degrees of the endpoints of this path in the caterpillar must be at least two (else they would not be part of the path as they would have been removed). 

\begin{definition}
A caterpillar is a tree in which all vertices are at most distance 1 from a central path.
\end{definition}

\begin{theorem}\label{thm:cat_bound}
If $G$ is a caterpillar with maximum degree $\Delta$, then $\Delta+1\leq \chi_{td}(G)\leq \Delta+3$.
\end{theorem}
\begin{proof}

Let $G$ be a caterpillar with maximum degree $\Delta$. Notice that $K_{1,\Delta}$ is a subgraph of $G$ so Proposition~\ref{prop:subgraph} and Theorem~\ref{thm:star} imply $\chi_{td}(G)\geq \chi_{td}(K_{1,\Delta}) \geq \Delta+1$. If $\Delta$ is $1$ or $2$, then note that $G$ is a path and we can refer to Theorem~\ref{thm:path} to determine $\chi_{td}(G)$. Therefore, we assume $\Delta\geq 3$ throughout the proof.

If there are multiple options for the choice of central path in $G$, choose one arbitrarily and call it $P$. We denote the vertices on $P$ by $v_1,v_2,\ldots,v_n$ consecutively, so that $deg(v_1)=deg(v_n)=1$ and $deg(v_i)>1$ for $i\neq 1, n$. 
We label $v_i$ on $P$ with $1$ if $i\equiv1\pmod{3}$, with $\Delta+3$ if $i\equiv2\pmod{3}$, and with $\Delta+2$ if $i\equiv0\pmod{3}$. We now consider labelings of the neighbors of an arbitrary vertex on $P$ based upon its label.

Suppose the vertex $v_i$ on $P$ gets label 1 so that $v_{i-1}$ and $v_{i+1}$ (if they exist) have labels $\Delta+2$ and $\Delta+3$, respectively. Then, since the subgraph induced by $v_i$ and its neighbors is a star $K_{1,d}$ for some $d\leq \Delta$, the remaining vertices adjacent to $v_i$ can be labeled with distinct integers from $[3,\Delta+1]$ without creating a double or triple.

Similarly, if $v_i$ gets label $\Delta+3$, then the neighbors of $v_i$ not on $P$ (of which there are at most $\Delta-2$) can be labeled arbitrarily with distinct integers from $[2,\Delta+1]$, excluding $\frac{\Delta+3}{2}$ if $\Delta$ is odd. Likewise, if $v_i$ gets label $\Delta+2$ we can label its neighbors not on $P$ arbitrarily with distinct integers from $[2,\Delta]$, excluding $\frac{\Delta+2}{2}$ if $\Delta$ is even. (Note that we avoid a $(\Delta+3,\Delta+2,\Delta+1)$-triple by disallowing the use of $\Delta+1$ as a label.)

Hence, by looking at all possible constructions we have proven that $\Delta+1\leq\chi_{td}(G)\leq\Delta+3$.
\end{proof}

In fact, we can classify caterpillars according to their total difference chromatic numbers, though the proofs are omitted here.

\begin{theorem}\label{thm:cat_delta_1}
If $G$ is a caterpillar, then $\chi_{td}(G) = \Delta+1$ if and only if 
\begin{enumerate}
    \item $\Delta$ is even,
    \item the distance between vertices of degree $\Delta$ is at least 3,
    \item no three consecutive vertices have degrees at least $\Delta-1$, and
    \item there are no five consecutive vertices with first and last having degree $\Delta$ and second and fourth having degree $\Delta-1$.
\end{enumerate}
\end{theorem}

\begin{theorem}\label{thm:cat_delta_3}
If $G$ is a caterpillar, $\chi_{td}(G) = \Delta+3$ if and only if $\Delta$ is odd and there are at least 3 vertices with degree $\Delta$ in a row.
\end{theorem}

Using Theorems~\ref{thm:cat_bound}, \ref{thm:cat_delta_1}, and \ref{thm:cat_delta_3} we can determine the total difference chromatic number of any caterpillar. In particular, any caterpillar not satisfying the conditions set forth in Theorem~\ref{thm:cat_delta_1} or \ref{thm:cat_delta_3} must have total difference chromatic number $\Delta+2.$

We now turn our attention to the trees known as lobsters, a natural extension of caterpillars.

\begin{definition}
A \emph{lobster} is a tree in which all vertices are at most distance 2 from a central path.
\end{definition}

In other words, removing the leaves of a lobster forms a caterpillar. 

We call the vertices on the central path the \emph{primary vertices}. Vertices adjacent to primary vertices which are not themselves primary are called \emph{secondary vertices}, and leaves adjacent to secondary vertices are called \emph{tertiary vertices}. We define $\Delta_1$ to be the maximum degree of the primary vertices and $\Delta_2$ to be the maximum degree of the secondary level vertices. A \emph{maximal lobster} is one in which every primary vertex with degree at least $2$ has degree $\Delta_1$ and every secondary vertex has degree $\Delta_2$. We will prove an upper bound on the total difference chromatic number of maximal lobsters and hence, by Proposition~\ref{prop:subgraph}, all lobsters.

The first step in determining our upper bound for the total difference chromatic number of an arbitrary maximal lobster is to label the central path, $P$ in the same way we labeled it for a caterpillar. Namely, we label $v_i$ on $P$ with $1$ if $i\equiv1\pmod{3}$, with $\Delta_1+3$ if $i\equiv2\pmod{3}$, and with $\Delta_1+2$ if $i\equiv0\pmod{3}$. Let $R=\{1,\Delta_1+2,\Delta_1+3\}$
By Theorem~\ref{thm:cat_bound} we know that all of the secondary vertices can have labels from $[1,\Delta_1+3]$. To avoid triples involving two primary and one secondary vertex, we restrict our possible secondary vertex labels to the set $S=[2,\Delta_1+1]$.

For each possible pair $(r,s)\in R\times S$ we consider the set of tertiary vertices adjacent to a secondary vertex with label $s$, which itself is adjacent to a primary vertex with label $r$. We greedily label this set of vertices and record the maximum label used. For given values of $\Delta_1$ and $\Delta_2$, we call this value $m_{\Delta_1,\Delta_2}(r,s)$. (See Table~\ref{tab:lobster} for an example with $\Delta_1=8$ and $\Delta_2=7$.) Notice that $m_{\Delta_1,\Delta_2+1}(r,s)\geq m_{\Delta_1,\Delta_2}(r,s)$+1.

\begin{table}
\centering
\begin{tabular}{|c|*{8}{c|}}\hline
\diagbox[width=2em]{r}{s}
&\makebox[2em]{2}&\makebox[2em]{3}&\makebox[2em]{4}
&\makebox[2em]{5}&\makebox[2em]{6}&\makebox[2em]{7}&\makebox[2em]{8}&\makebox[2em]{9}\\\hline
1 &&11&12&13&14&15&12&7\\\hline
10 &9&11&12&&14&15&12&\\\hline
11 &9&10&12&13&14&15&12&6\\\hline
\end{tabular}
\vskip0.5em
    \caption{The values of $m_{8,7}(r,s)$ for all $(r,s)\in R\times S$. The table entries for invalid $(r,s)$ pairs are left blank.}
    \label{tab:lobster}
\end{table}

\begin{definition}
We call the least value of $\Delta_2$ at which all subsequent unit increases of $\Delta_2$ cause unit increases of $m_{\Delta_1,\Delta_2}(r,s)$ the \emph{stabilization point} of $(r,s)$. 
\end{definition}

\begin{lemma}\label{lem:lob_stabilize}
Let $G$ be a lobster and let $r\in R=\{1,\Delta_1+2,\Delta_1+3\}$ and $s\in S=[2,\Delta_1+1]$. 
\begin{enumerate}
    \item[(a)] If $2\leq s<\frac{\Delta_1+4}{2}$, then $m_{\Delta_1,\Delta_1+3-s}(r,s)=\Delta_1+4$, and the stabilization point of $(r,s)$ is at most $\Delta_1+3-s$.
    \item[(b)] If $\frac{\Delta_1+4}{2}\leq s\leq
    \Delta_1+1$, then $m_{\Delta_1,s}(r,s)=2s+1$, and the stabilization point of $(r,s)$ is at most $s$.
\end{enumerate}
\end{lemma}

\begin{proof}
We first prove part (b). Suppose that $\frac{\Delta_1+4}{2}\leq s\leq\Delta_1+1$, and consider a secondary vertex, $v$, with label $s$ and degree $s$ adjacent to a primary vertex with label $r\in\{1,\Delta_1+2,\Delta_1+3\}$. We can label $s-1$ of the $s$ vertices adjacent to $v$ by using exactly one element from each set $\{i,2s-i\}$ for $1\leq i\leq 2s-1$. Notice that the primary vertex adjacent to $v$ has a label from one of these sets, since  $1$, $\Delta_1+2$, and $\Delta_1+3$ are all less than $2s$. The remaining tertiary vertex adjacent to $v$ cannot be labeled with any unused positive integer less than $2s$ since that would create a $(2s-i,s,i)$-triple, and labeling it with $2s$ would create a double. Hence, the least possible label for this one remaining vertex is $2s+1$, so $m_{\Delta_1,s}(r,s)=2s+1$. Notice that tertiary vertices adjacent to $v$ can be labeled with any integer $2s+k$ for $k>0$: these labels are larger than $2s$, and hence cannot form a double, and larger than $2s-1$ and hence cannot form a triple involving $s$. Therefore, the stabilization point of $(r,s)$ is at most $s$.

We prove part (a) similarly to the part (b). Again consider a secondary vertex $v$ with label $s$ and degree $\Delta_1+3-s$, for some $s$ with $2\leq s<\frac{\Delta_1+4}{2}$. We can label $s-1$ of the vertices adjacent to $v$ as before: by using exactly one element from each set $\{i,2s-i\}$ for $1\leq i\leq s-1$. We still must label $\Delta_1+3-s-(s-1)=\Delta_1+4-2s$ vertices adjacent to $v$, but this can be done using all integers in $[2s+1,\Delta_1+4]$ (if $r\neq 1$ then $r$ is part of this set). Observe that $\Delta_2=\Delta_1+3-s$ is greater than or equal to the stabilization point of $(r,s)$ since if $\Delta_2>\Delta_1+3-s$ we can label additional tertiary vertices with integers greater than $\max\{2s,r\}$.
\end{proof}

\begin{corollary}\label{cor:lob_less}
Assume $\Delta_2\geq\Delta_1+1$. Then
\begin{enumerate}
    \item[(a)] $m_{\Delta_1,\Delta_2}(r,s)\geq m_{\Delta_1,\Delta_2}(r,s')$ if $s\geq s'$, and
    \item[(b)] $m_{\Delta_1,\Delta_2}(r,s)=m_{\Delta_1,\Delta_2}(r',s)$ for any $r,r'\in R$,
\end{enumerate} when these values exist.
\end{corollary}

\begin{proof}
By Lemma~\ref{lem:lob_stabilize}, we know that $\Delta_2=\Delta_1+1$ is greater than or equal to the stabilization points of all $(r,s)\in R\times S$, and hence we can compute each $m_{\Delta_1,\Delta_2}(r,s)$ exactly when $\Delta_2\geq\Delta_1$. It is then straightforward to check that both parts of the claim hold.
\end{proof}

\begin{theorem}
For any lobster $G$ which is not a path or caterpillar,\\ $\Delta+1\leq\chi_{td}(G)\leq \Delta_1+\Delta_2+1$.
\end{theorem}

\begin{proof}
To prove that the claimed lower bound holds, we observe that a lobster contains as an induced subgraph the star $K_{1,\Delta}$ and then appeal to Proposition~\ref{prop:subgraph}.

Consider a maximal lobster in which all primary vertices (except the two of degree 1) have degree $\Delta_1$ and all secondary vertices have degree $\Delta_2$. First assume $\Delta_2=\Delta_1+1$. By Lemma~\ref{lem:lob_stabilize} and Corollary~\ref{cor:lob_less}(a), when $\Delta_2=\Delta_1+1$, the largest tertiary vertex label occurs when $s=\Delta_1+1$ and $r=1,\Delta_1+3$ and is $2\cdot(\Delta_1+1)+1=\Delta_1+\Delta_2+2$. 
By Lemma~\ref{lem:lob_stabilize}, all pairs $(r,s)$ have reached their stabilization points. Therefore, when $\Delta_2\geq\Delta_1+1$ and $k\geq 0$, $m_{\Delta_1,\Delta_2+k}(r,s)=m_{\Delta_1,\Delta_2}(r,s)+k$ and  $m_{\Delta_1,\Delta_2-k}(r,s)\leq m_{\Delta_1,\Delta_2}(r,s)-k$. Hence for any value of $\Delta_2$, the largest tertiary vertex label is at most $\Delta_1+\Delta_2+2$. 

 We now show that we can decrease this upper bound by 1. For the primary vertices with degree greater than 1, we know, by our construction, the labels of the two adjacent primary vertices. Therefore we need only $\Delta_1-2$ distinct secondary vertex labels (i.e., values of $s$), and hence, for each $r\in\{1,\Delta_1+2,\Delta_1+3\}$ we can choose the least $\Delta_1-2$ values of $s\in[2,\Delta_1+1]$ that do not create any doubles or triples.

We check all possible pairs $(r,s)\in R\times S$ for doubles and triples.  If $r=1$, $s$ cannot be $2$, or we would have a double. If $r=\Delta_1+2$, then $s$ cannot be $\Delta_1+1$ otherwise we would have a $(\Delta_1+3, \Delta_1+2, \Delta_1+1)$-triple involving two primary vertices. Finally, depending on the parity of $\Delta_1$, we must avoid either $s=\frac{\Delta_1+2}{2}$ (for $r=\Delta_1+2$) or $s=\frac{\Delta_1+3}{2}$ (for $r=\Delta_1+3$). See the empty boxes in Table~\ref{tab:lobster} for $(r,s)$ pairs that create doubles or triples in the case where $\Delta_1$ is even.

For each value of $r$, we need only the least possible $\Delta_1-2$ values of $s$ as secondary vertex labels. If $r=1$, since the only forbidden value of $s$ is 2, we can use all of the integers in $[3,\Delta_1]$. For one of the remaining two possible values of $r$ we cannot have $s=\frac{r}{2}$. For this value of $r$, we use as our secondary labels all integers in $[2,\Delta_1]$ except $\frac{r}{2}$. For the other value of $r$, we can use all integers in $[2,\Delta_1-1]$. Notice that we need not label any of the secondary vertices in our maximal lobster with $\Delta_1+1$.

Hence, when $\Delta_2=\Delta_1+1$, the largest relevant secondary vertex label is $s=\Delta_1$ and $m_{\Delta_1,\Delta_1+1}(r,\Delta_1)=2\cdot\Delta_1+2=\Delta_1+\Delta_2+1$ by Lemma~\ref{lem:lob_stabilize}. As all pairs $(r,s)$ have reached their stabilization points we know that for any value of $\Delta_2$, and for any lobster, the largest tertiary vertex label is at most $\Delta_1+\Delta_2+1$.
\end{proof}

We conclude by finding bounds for the total difference  chromatic number of an arbitrary tree in terms of its maximum degree. We first define a \emph{uniform full $\Delta$-ary tree}, denoted $T_{\Delta,h}$. For any integer $\Delta\geq 2$, $T_{\Delta,1}$ is defined to be the star $K_{1,\Delta}$. We let $v_0$ be the vertex in $T_{\Delta,1}$ with maximal degree, and each other vertex is a leaf. For each integer $h\geq 2$ we define $T_{\Delta,h}$ to be the tree obtained from $T_{\Delta,h-1}$ by appending $\Delta-1$ new leaves to each leaf of $T_{\Delta,h-1}$. Therefore, $T_{\Delta,h}$ is a rooted tree in which the distance from the root $v$ to every other vertex is at most $h$. It is maximal in the sense that every vertex whose distance from $v$ is less than $h$ has degree $\Delta$, while those at distance $h$ have degree 1. 

Notice that every tree is a subgraph of $T_{\Delta,h}$ for some choice of $\Delta$ and $h$. Therefore, if we find an upper bound for $\chi_{td}(T_{\Delta,h})$ the same bound works for an arbitrary tree $T$ with maximum degree $\Delta$ and appropriately-chosen $h$.

Before providing such a bound, we determine the total difference chromatic number for  uniform full $\Delta$-ary trees with height $2$.

\begin{lemma}\label{lem:tree2}
For a uniform full $\Delta$-ary tree with height $2$, \[\chi_{td}(T_{\Delta,2})=\left\lfloor\frac{3\Delta+3}{2}\right\rfloor.\]
\end{lemma}
\begin{proof}
The root $v_0$ of $T_{\Delta,2}$  is adjacent to $\Delta$ vertices, which we denote by $v_1,v_2,\ldots,v_\Delta$, each of which is adjacent to a further $\Delta$ vertices (including $v_0$). First assume $\Delta$ is odd. 
By Lemma~\ref{lem:star-realize}, for any $1\leq r\leq \Delta$ with $r\neq\frac{\Delta+3}{2}$, there are exactly $2r-2$ possible labels for each $v_i$ so that the maximum label on each $K_{1,\Delta}$ is exactly $\Delta+r$. 

Notice that $v_0,v_1,\ldots,v_\Delta$ all must have different labels since they are most distance 2 apart. Therefore, we must choose $r$ so that $2r-2\geq \Delta+1$. In particular we must have that $r\geq \frac{\Delta+3}{2}$. (In the exceptional $r=\frac{\Delta+3}{2}$ case, this inequality is still satisfied.)

If we label $v_0$ with $\Delta+r$, then each of these $\Delta+1$ copies of $K_{1,\Delta}$ will have maximum label $\Delta+r$ and therefore it is possible to give the entire $T_{\Delta,2}$ a $(\Delta+r)$-total difference labeling. 

To minimize $\Delta+r$, we choose $r=\frac{\Delta+3}{2}$. 
By construction, we cannot choose a smaller value of $r$ as there would not be enough distinct labels for $v_0,v_1,\ldots,v_\Delta$. Therefore, for odd $\Delta$, $\chi_{td}(T_{\Delta,2})=\Delta+\frac{\Delta+3}{2}=\frac{3\Delta+3}{2}$.

If $\Delta$ is even, Lemma~\ref{lem:star-realize} implies that there are $2r-1$ options for the labels of $v_0,v_1,\ldots,v_\Delta$. Therefore $2r-1\geq \Delta+1$, or, equivalently, $r\geq \frac{\Delta+2}{2}$. A similar argument to the one in the previous case implies that, for even $\Delta$, $\chi_{td}(T_{\Delta,2})=\Delta+\frac{\Delta+2}{2}=\frac{3\Delta+2}{2}.$
Combining the two cases gives us the desired result.

\end{proof}

\begin{theorem}\label{thm:tree}
For any uniform full $\Delta$-ary tree $T_{\Delta,h}$ with $h\geq 2$, $\lfloor{\frac{3\Delta+3}{2}}\rfloor\leq\chi_{td}(T_{\Delta,h})\leq 2\Delta+1$.
\end{theorem}

\begin{proof}
The lower bound is a result of  Lemma~\ref{lem:tree2} and Proposition~\ref{prop:subgraph}. 

To prove the upper bound, first consider the subgraph $T'$ induced by the root $v_0$ and its neighbors (which we denote $v_1,v_2,\ldots,v_\Delta$). The labels on these vertices must all be distinct by Proposition~\ref{prop:lowerbound}. Choose arbitrarily a label $\ell\in[1,2\Delta+1]$ for $v_0$. We show that for any choice of $\ell$, at most $\Delta$ numbers in $[1,2\Delta+1]$ cannot be labels of other vertices of $T'$, which leave the remaining numbers, of which there are at least $\Delta$, free to label the $\Delta$ neighbors of $v_0$.    

First, suppose $\ell<\Delta+1$. Then, to avoid triples, we may use at most one element from each set $\{\ell-i,\ell+i\}$ for $1\leq i\leq \ell-1$ (though, if $\ell$ is even we must not use the element $\frac{\ell}{2}$ from the set $\{\frac{\ell}{2},\frac{3\ell}{2}\}$ to avoid a double). Further, we may not use the number $2\ell$ as a label. All other numbers are valid for labeling vertices of $T'$ and, as we have only ruled out $\ell\leq\Delta$ options, there are at least enough to label the remaining $\Delta$ vertices of $T'$. An analogous argument shows that if $\ell>\Delta+1$ we also have enough numbers to label all vertices of $T'$.

If $\ell=\Delta+1$ then, to avoid triples, we again may use at most one element from each set $\{\ell-i,\ell+i\}$ for $1\leq i\leq \ell-1$ (and again, one of these includes $\frac{\ell}{2}$ if $\ell$ is even). In this case, however, $2\ell\notin[1,2\Delta+1]$ so we still have enough numbers to label the remaining vertices of $T'$.

In each case, all vertices of $T'$ can be labeled; suppose that vertex $v_j$ gets label $\ell_j$ for $1\leq j\leq \Delta$. By the same argument as above, there are enough available numbers in $[1,2\Delta+1]$ to label all neighbors of each $v_j$. We can give these vertices labels and then repeat the process until the vertices of the entire tree $T_{\Delta,h}$ have been labeled with integers in $[1,2\Delta+1]$.
\end{proof}

The following corollary follows directly from Theorem~\ref{thm:tree}. 
\begin{corollary}
For any tree $T$, $\Delta+1\leq\chi_{td}(T)\leq 2\Delta+1$.
\end{corollary}

\bibliographystyle{plain}
\bibliography{references}

\begin{thebibliography}{1}

\bibitem{Byers2}
Zhenming Bi, Alexis Byers, Sean English, Elliot Laforge, and Ping Zhang.
\newblock Graceful colorings of graphs.
\newblock {\em J. Combin. Math. Combin. Comput.}, 101:101 -- 119, 2017.

\bibitem{Gallian}
Joseph~A. Gallian.
\newblock A dynamic survey of graph labeling.
\newblock {\em Electronic J. Combin.}, DS6, 2018.

\bibitem{Golomb}
Solomon~W. Golomb.
\newblock How to number a graph.
\newblock In R.~C. Read, editor, {\em Graph Theory and Computing}, pages
  23--37. Academic Press, New York, 1972.

\bibitem{Rosa}
Alexander Rosa.
\newblock On certain valuations of the vertices of a graph.
\newblock In {\em Theory of Graphs (Internat. Symposium, Rome, July 1966)},
  pages 349--355. Gordon and Breach, New York, 1967.

\end{thebibliography}

\end{document}